\documentclass[12pt]{amsart}
\usepackage[english]{babel}
\usepackage[T1]{fontenc}
\usepackage[latin1]{inputenc}
\usepackage{lmodern}
\usepackage{a4}
\usepackage[all]{xy}
\usepackage{verbatim} 
\usepackage{stmaryrd}
\usepackage{amsfonts}
\usepackage{amsmath,amssymb,amsthm}
\usepackage{enumerate}
\usepackage{xspace}
\usepackage{ulem}
\usepackage{euscript}
\usepackage{epsfig}
\usepackage{txfonts}
\usepackage{pxfonts}
\usepackage{babel,indentfirst}
\usepackage[colorlinks,linkcolor=red,citecolor=blue]{hyperref}

\makeatletter\@addtoreset{equation}{section}

\newtheorem{theorem}{Theorem}[section]
\newtheorem{corollary}[theorem]{Corollary}
\newtheorem{lemma}[theorem]{Lemma}
\newtheorem{proposition}[theorem]{Proposition}
\theoremstyle{remark}

\newtheorem{example}[theorem]{Example}

\newcommand{\C}{\mathbb{C}}
\newcommand{\R}{\mathbb{R}}
\newcommand{\N}{\mathbb{N}}

\newcommand{\lh}{{\mathcal B}(H)}

\renewcommand{\span}{{\rm Span}}

\newcommand{\ran}{{\rm Ran }  }

\renewcommand{\ker}{{\rm Ker }  }
\newcommand{\gsh}{\mathcal{GS}(H)}
\newcommand{\sh}{\mathcal{S}(H)}
\newcommand{\ess}{{ \rm M} }

\makeatletter
\@namedef{subjclassname@2020}{%
  \textup{2020} Mathematics Subject Classification}
\makeatother

\title[$C$-normality of rank-one perturbations of normal operators]
{$C$-normality of rank-one perturbations\\ of normal operators}

\author{Zouheir Amara}
\address{Department of Mathematics, Labo LIABM, FSO, Mohammed First University, 60000 Oujda, Morocco}
\email{z.amara@ump.ac.ma}

\author{Mourad Oudghiri}
\address{Department of Mathematics, Labo LIABM, FSO, Mohammed First University, 60000 Oujda, Morocco}
\email{m.oudghiri@ump.ac.ma}



\begin{document}

\maketitle

\begin{abstract}
For a separable complex Hilbert space $H$, we say that a bounded linear operator $T$ acting on $H$ is $C$-normal, where $C$ is a conjugation on $H$, if it satisfies $CT^*TC=TT^*$. For a normal operator, we give geometric conditions which guarantee that its rank-one perturbation is a $C$-normal for some conjugation $C$. We also obtain some new properties revealing the structure of $C$-normal operators.
\end{abstract}

\section{Introduction}

Denote by $H$ a separable complex Hilbert space of dimension greater than two, and by $\lh$ the algebra of all bounded linear operators acting on $H$. A conjugate-linear operator $C$ on $H$ is said to be {\it conjugation} if it satisfies the conditions:
\begin{enumerate}[(i)]
 \item $C$ is isometric : $\langle Cx, Cy\rangle = \langle y, x\rangle$ for all $x, y\in H$,
 \item $C$ is involutive : $C^2 = I$.
\end{enumerate}
We say that an operator $A\in\lh$ is {\it $C$-symmetric} if $A = CA^*C$, and it is called {\it complex symmetric} if it is $C$-symmetric for some conjugation $C$ on $H$. Complex symmetric operators are exactly those that have a symmetric (i.e. self-transpose) matrix representation relative to some orthonormal basis, see \cite{GaPu.symmI}.
Their general study was initiated by Garcia, Putinar, and Wogen, see \cite{GaPu.symmI,GaPu.symmII,gar-wog1,Ga.some}. The class of complex symmetric operators is large. It contains normal operators, bi-normal operators, quadratic operator, Hankel operators, truncated Toeplitz operators and many standard integral operators such as the Volterra integration operator, see \cite{Ga.some}.

Following \cite{wicher}, an operator $A\in\lh$ is said to be {\it $C$-normal} if 
$$CA^*AC=AA^*,$$ 
or equivalently, if $C\vert A\vert C=\vert A^*\vert$ where $\vert A\vert:=\sqrt{A^*A}$. Clearly, if $A$ is $C$-normal then so is $A^*$, and every $C$-symmetric operator is $C$-normal.

Let us fix the following notations:
$$
\sh=\left\lbrace A\in\lh : \mbox{$\exists$ a conjugation $C$ on $H$ such that $CAC=A^*$}\right\rbrace
$$
and
$$
\gsh=\left\lbrace A\in\lh : \mbox{$\exists$ a conjugation $C$ on $H$ such that $CA^*AC=AA^*$}\right\rbrace.
$$

It is obvious that $\sh\subset\gsh$, but, as will be shown by Example \ref{ExampleIntro}, in general this inclusion is strict. However, it is proved in \cite{w-2z} that $C$-normality and $C$-symmetry coincide on a dense class of operators in $\lh$.

In terms of matrices, $\gsh$ is the set of all operators $A$ such that the matrices of $A^*A$ and $AA^*$ are transposed relative to some orthonormal basis. In fact, this can easily be obtained using the fact that each conjugation $C$ has a fixed orthonormal basis $\lbrace e_i\rbrace$; i.e. $Ce_i=e_i$ for every $i\geq 1$ (see \cite{GaPu.symmI}). 

For more details about the structure and the properties of $C$-normal operators, the reader is referred to \cite{Cnor1,wicher,Cnor2,w-2z}.

In this paper, we provide geometric conditions on a normal operator and a rank-one operator so that their sum is a $C$-normal for some conjugation $C$. This shows that if $N\in\lh$ is normal, then $N+R\in\gsh$ for a large class of rank-one operators $R$. We derive an example to show that these conditions do not guarantee that the obtained operator is complex symmetric. We also obtain a new characterization of operators belonging to $\gsh$; in fact, we show that the property {``$T$ being $C$-normal for some conjugation $C$''} can be reduced to an equality involving operators that act on a Hilbert space, smaller than $H$, denoted by $\ess(T)$ (Theorem \ref{ess}).

\medskip

The main result of this paper is stated in the second section whereas the third one is devoted to its proof.

\section{Main results}

For non-zero vectors $u,v\in H$, we denote by $u\otimes v$ the rank one operator given by $(u\otimes v)(x)=\langle x,v\rangle u$ for all $x\in H$. Note that all rank-one operators have such representation.

Let $N \in \lh$ be a normal operator. In \cite{Ga.some}, the authors proved that if $U\in\lh$ is unitary and limit (in strong operator topology) of operators of the form $P(N,N^*)$ with $P\in\C[X,Y]$, then $N+\lambda Ux\otimes x\in\sh$ for all $x \in H$ and $\lambda \in \C$. Later in \cite{A.O}, this result is shown to remain valid to all unitary operators commuting with $N$. Since $\gsh$ contains all complex symmetric operators, then one may expect that a larger class of rank-one perturbations of normal operators must lie in $\gsh$. This is indeed the case.

For a normal operator $N$, $E_{N}$ denotes the spectral measure associated with $N$.

The main theorem of this paper is the following.

\begin{theorem}\label{main.t}
Let $N\in\lh$ be normal, and let $x$ and $y$ be two vectors in $H$. If 
\begin{enumerate}[\rm (i)]
\item $\langle NE_{\vert N\vert}(\Delta)x,x\rangle=\langle NE_{\vert N\vert}(\Delta)y,y\rangle$,
\item $\langle E_{\vert N\vert}(\Delta)x,x\rangle=\langle E_{\vert N\vert}(\Delta)y,y\rangle$,
\end{enumerate}
for every borel subset $\Delta\subset\R^+$, then
$$
N+\lambda y\otimes x \in\gsh\quad\mbox{for every $\lambda\in\C.$}
$$
\end{theorem}

For unitary operators, the previous theorem can be reformulated as follows.

\begin{theorem}
Let $U\in\lh$ be a unitary operator, and let $x$ and $y$ be two vectors in $H$ having the same norm and such that $\langle Ux,x\rangle=\langle Uy,y\rangle$. Then 
$$
U+\lambda y\otimes x \in\gsh\quad\mbox{for every $\lambda\in\C.$}
$$
\end{theorem}

\begin{proof} The proof follows from the fact that $U$ satisfies the conditions of the previous theorem. Indeed, since $\vert U\vert=I$, we have either $E_{\vert U\vert}(\Delta)=0$ or $E_{\vert U\vert}(\Delta)=I$ for every Borel subset $\Delta$ of $\R^+$.
\end{proof}

As a consequence of the previous theorem, we derive the following corollary.

\begin{corollary}
Let $U$ and $V$ be unitary operators acting on a complex separable Hilbert space $K$, and let $R\in\mathcal{B}(K)$ be a rank-one operator. Then the operators
$$
\begin{pmatrix}
R & V\\
U & 0
\end{pmatrix},
\quad
\begin{pmatrix}
0 & V+R\\
U & 0
\end{pmatrix},
\quad
\begin{pmatrix}
0 & V\\
U+R & 0
\end{pmatrix}
\quad
\mbox{and}
\quad
\begin{pmatrix}
0 & V\\
U & R
\end{pmatrix}
$$
belong to $\mathcal{GS}(K\oplus K)$.
\end{corollary}

\begin{proof}
Write
$$
A=\begin{pmatrix}
R & V\\
U & 0
\end{pmatrix},\quad
W=\begin{pmatrix}
0 & V\\
U & 0
\end{pmatrix}\quad\mbox{and}\quad
F=
\begin{pmatrix}
R & 0\\
0 & 0
\end{pmatrix}.
$$
Then $W$ is unitary and $F=\| F\| y\otimes x$ where $x$ and $y$ are unit vectors in $K\oplus 0$. As $\langle Wx,x\rangle=0=\langle Wy,y\rangle$, we get by the previous theorem that $A=W+F\in\mathcal{GS}(K\oplus K)$. Similarly, we prove that the other
operators are in $\mathcal{GS}(K\oplus K)$.
\end{proof}

The following example shows that the conditions of Theorem \ref{main.t} do not guarantee that the perturbed  normal operator is complex symmetric.

\begin{example}\label{ExampleIntro}
Let $A\in\mathcal{B}(\C^4)$ be the operator represented in the canonical basis $\lbrace e_i\rbrace$ by the following matrix
$$
A=\begin{pmatrix}
1 & 0 & 0 & 1\\
-1 & 0 & 1 & 0\\
1 & 0 & 0 & 0\\
0 & 1 & 0 & 0
\end{pmatrix}.
$$
Then $A=U + \sqrt{2} y\otimes x$ where
$$
U=\begin{pmatrix}
0 & 0 & 0 & 1\\
0 & 0 & 1 & 0\\
1 & 0 & 0 & 0\\
0 & 1 & 0 & 0
\end{pmatrix}
$$
is a unitary operator, and $x=e_1$ and $y=\sqrt{2}^{-1}(e_1-e_2)$ are unit vectors satisfying the conditions of Theorem \ref{main.t}. If we let $trace$ denote the trace of an operator, then we have
$$
trace\left( A^2(A{A^*}^2-{A^*}^2A)A^2A^* \right)=4,
$$
and so $A$ is not complex symmetric by \cite[Theorem 1]{gapoten}.
\end{example}

We note that the second condition in Theorem \ref{main.t} cannot be relaxed. This can be seen from the next example.

\begin{example}
Consider the operator $A\in\mathcal{B}(\C^3)$ having the following matrix representation in the canonical basis $\lbrace e_i\rbrace$ 
$$
A=\begin{pmatrix}
1 & 0 & 1\\
0& -1& 1\\
0& 0& 0
\end{pmatrix}.
$$
Write $A=N+\sqrt{2}y\otimes x$ where
$$
N=\begin{pmatrix}
1 & 0 & 0\\
0& -1& 0\\
0& 0& 0
\end{pmatrix},
$$
$x=e_3$ and $y=\sqrt{2}^{-1}(e_1+e_2)$. Clearly, $N$ is normal, and the spectral projections corresponding to $\vert N\vert$ are $E_1=I$, $E_2=I-e_3\otimes e_3$, $E_3=e_3\otimes e_3$ and $E_4=0$. Furthermore, 
\begin{eqnarray*}
0=\langle NE_1 x,x\rangle=\langle NE_1 y,y\rangle &=&\langle NE_2 x,x\rangle=\langle NE_2 y,y\rangle\\
&=& \langle NE_3 x,x\rangle=\langle NE_3 y,y\rangle,\\
&=& \langle NE_4 x,x\rangle=\langle NE_4 y,y\rangle,
\end{eqnarray*}
and hence the first condition in Theorem \ref{main.t} is fulfilled, which is not the case of the second one because $0=\langle E_2 x,x\rangle\neq\langle E_2 y,y\rangle=1$.

Let us show that $A\notin\mathcal{GS}(\C^3)$. Assume that $A$ is $C$-normal for some conjugation $C$ on $\C^3$. Elementary calculations show that
$$
A^*A=
\begin{pmatrix}
1 & 0 & 1\\
0& 1& -1\\
1& -1& 2
\end{pmatrix},\quad
AA^*=
\begin{pmatrix}
2 & 1 & 0\\
1& 2& 0\\
0& 0& 0
\end{pmatrix},
$$

$$
\ker(A^*A)=\span\lbrace f_1\rbrace,\quad \ker(A^*A-3)=\span\lbrace f_2\rbrace
$$
$$
\ker(AA^*)=\span\lbrace g_1\rbrace\quad\mbox{and}\quad \ker(AA^*-3)=\span\lbrace g_2 \rbrace
$$
where 
$$
f_1= \sqrt{3}^{-1}(e_1-e_2-e_3), f_2=\sqrt{6}^{-1}(e_1-e_2+2e_3), g_1=e_3 \mbox{ and } g_2=\sqrt{2}^{-1}(e_1+e_2).
$$
Since the above vectors are unit, the equalities
$$
CA^*AC=AA^*\quad\mbox{and}\quad C(A^*A-3)C=AA^*-3
$$
ensure the existence of unimodular scalars $\alpha_1,\alpha_2\in\C$ such that $Cf_1=\alpha_1 g_1$ and $Cf_2=\alpha_2 g_2$.
Hence,
$$
0=\vert\langle f_1,g_2\rangle\vert=\vert\langle Cg_2, Cf_1\rangle\vert=\vert\langle f_2,g_ 1\rangle\vert=2\sqrt{6}^{-1},
$$
a contradiction.
\end{example}

The following example illustrates that the first condition also is indispensable in Theorem \ref{main.t}.

\begin{example}
Let $U\in\mathcal{B}(\C^3)$ be the unitary operator given by
$$
U=\begin{pmatrix}
i & 0 & 0\\
0& -1& 0\\
0& 0& 1
\end{pmatrix},
$$
and set $x=e_3$ and $y=\sqrt{2}^{-1}(e_1+e_2)$. Then, it is easy to see that $x$ and $y$ satisfy the second condition of Theorem \ref{main.t}. On the other hand, the first one is not satisfied because
$$
1=\langle UE_{\vert U\vert}(\lbrace 1\rbrace)x,x\rangle\neq\langle UE_{\vert U\vert}(\lbrace 1\rbrace)y,y\rangle=2^{-1}(i-1).
$$
Now we show that
$$
A:=U+\sqrt{2}y\otimes x=
\begin{pmatrix}
i & 0 & 1\\
0& -1& 1\\
0& 0& 1
\end{pmatrix}\notin\mathcal{GS}(\C^3).
$$
Calculations yield that
$$
A^*A=\begin{pmatrix}
1 & 0 & -i\\
0& 1& -1\\
i& -1& 3
\end{pmatrix}
\quad\mbox{and}\quad
AA^*=\begin{pmatrix}
2 & 1 & 1\\
1& 2& 1\\
1& 1& 1
\end{pmatrix}
$$
$$
\ker(A^*A-1)=\span\lbrace f_1\rbrace,\quad \ker\left(A^*A-(\sqrt{3}+2)\right)=\span\lbrace f_2\rbrace
$$
$$
\ker(AA^*-1)=\span\lbrace g_1\rbrace\quad\mbox{and}\quad \ker\left(AA^*-(\sqrt{3}+2)\right)=\span\lbrace g_2 \rbrace
$$
where
$$
f_1=\sqrt{2}^{-1}(e_1+ie_2),\quad  f_2=\sqrt{(6+2\sqrt{3})}^{-1}\left(e_1-ie_2+(\sqrt{3}+1)ie_3\right),
$$
$$
g_1= \sqrt{2}^{-1}(e_1-e_2),\quad\mbox{and}\quad g_2=\sqrt{(6-2\sqrt{3})}^{-1}\left(e_1+e_2+(\sqrt{3}-1)e_3\right).
$$ 
Hence if $A$ is $C$-normal for some conjugation $C$, then we can prove, as in the previous example, that
$$
\sqrt{(6-2\sqrt{3})}^{-1}=\vert\langle f_1,g_2\rangle\vert=\vert\langle f_2,g_1\rangle\vert=\sqrt{(6+2\sqrt{3})}^{-1},
$$
which is a contradiction.
\end{example}

Recall that an operator $V$ acting on a Hilbert space is called {\it partial isometry} if $\| Vx\|=\| x\|$ for every $x\in\ker(V)^\perp$, or equivalently, if $V^*V$ is the projection onto $\ker(V)^\perp$.

\begin{corollary}
Let $N\in\lh$ be normal, and let $V\in\lh$ be a partial isometry that commutes with $N$. Then
$$
N+\lambda Vx\otimes x \in\gsh\quad\mbox{for all $x\in \ker(V)^\perp$ and $\lambda\in\C$}.
$$
\end{corollary}

\begin{proof}
As $V$ commutes with $N$, and hence with $N^*N$ by Fuglede's theorem, it follows that $V$ commutes with all the spectral projections of $\vert N\vert$. Hence, for every Borel subset $\Delta\subset\R^+$, we have
\begin{eqnarray*}
\langle NE_{\vert N\vert}(\Delta)Vx,Vx\rangle &=&\langle VNE_{\vert N\vert}(\Delta)x,Vx\rangle\\
&=&\langle NE_{\vert N\vert}(\Delta)x,V^*Vx\rangle\\
&=&\langle NE_{\vert N\vert}(\Delta)x,x\rangle,
\end{eqnarray*}
and similarly, $\langle E_{\vert N\vert}(\Delta)Vx,Vx\rangle=\langle E_{\vert N\vert}(\Delta)x,x\rangle$. The proof follows now from Theorem \ref{main.t}.
\end{proof}

\begin{corollary}
Let $N\in\lh$ be normal, and let $R\in\lh$ be a rank-one operator that commutes with $N$. Then
$N+R\in\gsh$.
\end{corollary}
\begin{proof}
Set $R=\lambda x\otimes y$ where $x$ and $y$ are unit vectors. Let $\Delta\subset \R^+$ be a Borel subset. Since $R$ commutes with $N$, we get
$$
(x\otimes y)E_{\vert N\vert}(\Delta)=E_{\vert N\vert}(\Delta)(x\otimes y)\quad\mbox{and}\quad
(x\otimes y)NE_{\vert N\vert}(\Delta)=NE_{\vert N\vert}(\Delta)(x\otimes y).
$$
It follows that
\begin{eqnarray*}
\langle E_{\vert N\vert}(\Delta) x ,x\rangle &=&\Vert E_{\vert N\vert}(\Delta) x\Vert^2
=\Vert (E_{\vert N\vert}(\Delta)x)\otimes y\Vert^2=\Vert E_{\vert N\vert}(\Delta)(x\otimes y)\Vert^2\\
&=& \Vert (x\otimes y)E_{\vert N\vert}(\Delta)\Vert^2=\Vert x\otimes (E_{\vert N\vert}(\Delta)y)\Vert^2=\langle E_{\vert N\vert}(\Delta) y ,y\rangle
\end{eqnarray*}
and
\begin{eqnarray*}
\langle NE_{\vert N\vert}(\Delta) x ,x\rangle &=&\langle NE_{\vert N\vert}(\Delta) (x\otimes y)y,x\rangle=\langle (x\otimes y)NE_{\vert N\vert}(\Delta)y,x\rangle\\
&=&\langle x\otimes (E_{\vert N\vert}(\Delta)N^*y)y,x\rangle=\langle y,E_{\vert N\vert}(\Delta)N^*y\rangle\langle x,x\rangle\\
&=&\langle NE_{\vert N\vert}(\Delta) y,y\rangle.
\end{eqnarray*}
Therefore, $N+R\in\gsh$ by Theorem \ref{main.t}.
\end{proof}

The proof of the main result, which will be given at the next section, requires some preliminaries.
We begin with the following technical lemma.

\section{Proof of the main result}

\begin{lemma}\label{conj.angl}
Let $U\in\lh$ be a unitary operator, and let $x$ and $y$ be vectors in $H$ having the same norm and such that $\langle Ux,x\rangle=\langle Uy,y\rangle$. Then there is a conjugation $C$ on $H$ satisfying $Cx=y$ and $CUx=U^*y$.
\end{lemma}

\begin{proof}
Obviously, there is no loss of generality in assuming that $x$ and $y$ are unit vectors.
Let $a$ and $b$ be unit vectors in $x^\perp$ and $y^\perp$, respectively, such that
$$
Ux=\alpha x + \beta a\quad\mbox{and}\quad U^*y=\tilde{\alpha} y + \tilde{\beta} b
$$
where $\beta,\tilde{\beta}\geq0$. Then
$$
\alpha=\langle Ux,x\rangle=\langle Uy,y\rangle=\overline{\langle U^*y,y\rangle}=\overline{\tilde{\alpha}},
$$
and
$$
\tilde{\beta}=\sqrt{1-\vert \tilde{\alpha}\vert^2}=\sqrt{1-\vert \alpha\vert^2}=\beta.
$$
Note that if $\beta=0$ then, by \cite[Theorem 2.1]{zhu.guang}, there exists a conjugation $C$ on $H$ such that $Cx=y$, and whence $CUx=C(\alpha x)=\overline{\alpha}y=U^*y$.

Suppose now that $\beta$ is nonzero. Since
$$
\langle x,b\rangle=\langle x,\beta^{-1}(U^*y-\overline{\alpha}y)\rangle=\beta^{-1}\left(\langle x,U^*y\rangle-\alpha\langle x,y\rangle\right)=\beta^{-1}\left(\langle Ux,y\rangle-\alpha\langle x,y\rangle\right)
$$
and
$$
\langle a,y\rangle=\beta^{-1}\langle(Ux-\alpha x),y\rangle=\beta^{-1}\left(\langle Ux,y\rangle-\alpha\langle x,y\rangle\right),
$$
we have $\langle x,b\rangle=\langle a,y\rangle$, whence, again by \cite[Theorem 2.1]{zhu.guang}, there exists a conjugation $C$ on $H$ satisfying $Cx=y$ and $Ca=b$. This completes the proof.
\end{proof}

For an operator $A\in\lh$, we let {\it $\ess(A)$} denote the orthogonal of the largest subspace $F$ of $H$ that satisfies the following conditions:
\begin{enumerate}[\rm (i)]
\item $F$ is invariant under $\vert A\vert$ and $\vert A^*\vert$,
\item $\vert A\vert_{\vert F}=\vert A^*\vert_{\vert F}$.
\end{enumerate}

For a subset $F\subseteq H$, we denote by $\overline{F}$ its topological closure.

\begin{proposition}
Let $A\in\lh$. Then 
$$
\ess(A) = \overline{\sum_{n\geq 1} \ran(\vert A\vert^n-\vert A^*\vert^n}).
$$
\end{proposition}

\begin{proof}
Put
$$
F:= \left(\cup_{n\in\N^*}{\ran(\vert A\vert^n-\vert A^*\vert^n})\right)^\perp =\cap_{n\in\N^*}\ker(\vert A\vert^n-\vert A^*\vert^n).
$$
Clearly, $\vert A\vert_{\vert F}=\vert A^*\vert_{\vert F}$. To show that $F$ is invariant under $\vert A\vert$,  let $x\in F$ and $n\in\N^*$. Since $\vert A\vert x=\vert A^*\vert x$ and $\vert A\vert^{n+1} x=\vert A^*\vert^{n+1} x$, we have
\begin{eqnarray*}
(\vert A\vert^n-\vert A^*\vert^n)\vert A\vert x =\vert A\vert^n\vert A\vert x-\vert A^*\vert^n\vert A\vert x
= \vert A\vert^{n+1}x-\vert A^*\vert^{n+1}x
= 0,
\end{eqnarray*}
and so $\vert A\vert x\in F$. Similarly, we show that $F$ is invariant under $\vert A^*\vert$.

Now, let $L$ be an invariant subspace of $\vert A\vert$ and $\vert A^*\vert$ such that $\vert A\vert_{\vert L}=\vert A^*\vert_{\vert L}$. For every $x\in L$, we have
$$
\vert A\vert^n x =\vert A\vert^n_{\vert L}x=\vert A^*\vert^n_{\vert L}x=\vert A^*\vert^n x\quad\mbox{for all $n\in\N^*$},
$$
that is $x\in F$. Consequently, $L\subset F$. Thus, $\ess(A)=F^\perp$ .
\end{proof}

It should be noted that $\vert A\vert$ and $\vert A^*\vert$ are reduced by $\ess(A)$, and that they coincide on $\ess(A)^{\perp}$.

\begin{theorem}\label{ess}
Let $A\in\lh$. Then $A\in\gsh$ if and only if there exists a conjugation $J$ on $\ess(A)$ such that $J\vert A\vert_{\vert \ess(A)}J=\vert A^*\vert_{\vert \ess(A)}$.
\end{theorem}

\begin{proof}
Assume that $A$ is $C$-normal for some conjugation $C$ on $H$, and let $x\in \ess(A)^\perp$ and $n\in\N^*$. Then
\begin{eqnarray*}
(\vert A\vert^n-\vert A^*\vert^n)Cx &=&\vert A\vert^n Cx - \vert A^*\vert^n Cx
= C\vert A^*\vert^n x-C\vert A\vert^n x\\
&=&C(\vert A^*\vert^n x - \vert A\vert^n x)=0,
\end{eqnarray*}
and so $C\ess(A)^\perp\subset \ess(A)^\perp$. Since $C^2=I$, we get that $\ess(A)^\perp\subset C\ess(A)^\perp$, and hence $C\ess(A)^\perp = \ess(A)^\perp$. Using the fact that $C$ is an isometry, we get that $C\ess(A)= \ess(A)$.
Hence $J:=C_{\vert \ess(A)}$ is a conjugation on $\ess(A)$ and $J\vert A\vert_{\vert \ess(A)}J=\vert A^*\vert_{\vert \ess(A)}$.

Conversely, let $J$ be a conjugation on $\ess(A)$ such that $J\vert A\vert_{\vert \ess(A)}J=\vert A^*\vert_{\vert \ess(A)}$. Since $\vert A\vert_{\vert \ess(A)^\perp}$ is normal, there is a conjugation $E$ on $\ess(A)^\perp$ such that 
$$E\vert A\vert_{\vert \ess(A)^\perp}E=\vert A\vert_{\vert \ess(A)^\perp}.$$ 
As $\vert A\vert_{\vert \ess(A)^\perp}=\vert A^*\vert_{\vert \ess(A)^\perp}$, we get that $A$ is $C$-normal with respect to the conjugation 
$$
C=\begin{pmatrix}
J & 0\\
0 & E
\end{pmatrix}
\begin{array}{l}
\ess(A)\\
\ess(A)^\perp
\end{array},
$$
which completes the proof.
\end{proof}


It is proved in \cite[Proposition 4.3]{w-2z} that if $T=D\oplus A$ (orthogonal sum) with $D$ being a diagonal operator, then $T$ is $C$-normal if and only if $A$ is $J$-normal, where $C$ and $J$ are conjugations on the underlying Hilbert spaces. This result is extended to normal operators.

\begin{corollary}\label{N+A}
Let $N$ be a normal operator acting on a separable complex Hilbert space $K$, and let $A\in\lh$. Then the orthogonal sum $N\oplus A\in\mathcal{GS}(K\oplus H)$ if and only if $A\in\gsh$.
\end{corollary}

\begin{proof}
The sufficiency is obvious. Assume that $N\oplus A\in\mathcal{GS}(K\oplus H)$. It follows by the previous lemma that there exists a conjugation $J$ on $\ess(N\oplus A)$ such that 
\begin{equation}\label{eq.N+A}
J\vert N\oplus A\vert_{\vert \ess(N\oplus A)}J=\vert N^*\oplus A^*\vert_{\vert \ess(N\oplus A)}.
\end{equation}
Let $x=x_1+x_2$ where $x_1\in K$ and $x_2\in H$. We have
\begin{eqnarray*}
x\in \ess(N\oplus A)^\perp\quad &\Leftrightarrow & \quad \left( \vert N\oplus A\vert^n- \vert N^*\oplus A^*\vert^n\right)(x_1+x_2)=0,\quad\forall n \in\N^*\\
&\Leftrightarrow & \quad \left( \vert N\vert^n \oplus \vert A\vert^n- \vert N\vert^n \oplus \vert A^*\vert^n\right)(x_1+x_2)=0,\quad\forall n \in\N^*\\
&\Leftrightarrow& \quad\left( \vert A\vert^n- \vert A^*\vert^n\right)x_2=0,\quad\forall n \in\N^*\\
&\Leftrightarrow& \quad x\in K\oplus (H\ominus \ess(A)),
\end{eqnarray*}
and so $\ess(N\oplus A)=\ess(A)$. Therefore, by (\ref{eq.N+A}), we get
$$
J\vert A\vert_{\vert \ess(A)}J=\vert A^*\vert_{\vert \ess(A)},
$$
and thus $A\in\gsh$.

\end{proof}

\begin{lemma}\label{cnacn}
Let $A_n\in\lh$ be a sequence of $C_n$-normal operators. If $A_n$ converges to an operator $A\in\lh$, then $C_n\vert A\vert C_n$ converges to $\vert A^*\vert$. 
\end{lemma}

\begin{proof}
We have
\begin{eqnarray*}
\parallel C_n\vert A\vert C_n - \vert A^*\vert\parallel &=&\parallel C_n\vert A\vert C_n - C_n\vert A_n\vert C_n  + C_n\vert A_n\vert C_n - \vert A^*\vert\parallel\\
&\leq &\parallel C_n\vert A\vert C_n - C_n\vert A_n\vert C_n\parallel + \parallel \vert A_n^*\vert  - \vert A^*\vert\parallel\\
&\leq& \parallel \vert A\vert  - \vert A_n\vert \parallel + \parallel \vert A_n^*\vert  - \vert A^*\vert\parallel,
\end{eqnarray*}
and so $C_n\vert A\vert C_n$ converges to $\vert A^*\vert$.
\end{proof}

An operator $A\in\lh$ is called {\it essentially normal} if $A^*A-AA^*$ is a compact operator. Note that the class of essentially normal operators contains all normal operators and is invariant under compact perturbations.

\begin{lemma}\label{cmpctdiff}
If $A\in\lh$ is essentially normal, then the operator $\vert A\vert^m-\vert A^*\vert^m$ is compact for every positive integer $m$.
\end{lemma}

\begin{proof}
Note that it suffices to establish the result for $m=1$. Indeed, if $\vert A\vert=\vert A^*\vert+K$ with $K$ a compact operator then, for every $m\in\N^*$, we have
$$
\vert A\vert^m=(\vert A^*\vert+K)^m=\vert A^*\vert^m+K_m
$$
for some compact operator $K_m$.

Let $\lbrace P_n\rbrace$ be a sequence of polynomials that converges uniformly to $t\mapsto\sqrt{t}$
on $\sigma(A^*A)\cup\sigma(AA^*)$. Writing $A^*A=AA^*+K$ for some compact operator $K$, one can easily check that $P_n(A^*A)=P_n(AA^*)+K_n$ for some compact operator $K_n$.
Now, as the sequences $P_n(A^*A)$ and $P_n(AA^*)$ converge respectively to $\vert A\vert$ and $\vert A^*\vert$, we infer that $K_n$ converges to some compact operator. Thus,  $\vert A\vert-\vert A^*\vert$ is compact.
\end{proof}

From \cite{capro}, we recall the following useful result.

\begin{lemma}\cite[Corollary 7.2]{capro}
Let $A\in\lh$ be compact, and let $\lbrace C_n\rbrace$ be a sequence of conjugations on $H$ such that $C_n AC_n+A^*$ converges to zero as $n$ tends to infinity. If $P$ denotes the projection onto $\overline{\ran(A)+\ran(A^*)}$, then there exists a subsequence $\lbrace n_j\rbrace_j$ such that ${PC_{n_j}}_{\vert\ran(P)}$ converges to a conjugation on $\ran(P)$.
\end{lemma}

The proof of the following lemma is inspired by \cite[Theorem 7.3]{capro}.

\begin{lemma}\label{essnor}
Let $A\in\lh$ be essentially normal. Then $A\in\gsh$ if and only if $A\in\overline{\gsh}$. 
\end{lemma}

\begin{proof}
Assume that there is a sequence $A_n$ of $C_n$-normal operators that converges to $A$. Then, by Lemma \ref{cnacn}, $C_n\vert A\vert C_n$ and $C_n\vert A^*\vert C_n$ converge to $\vert A^*\vert$ and $\vert A\vert$, respectively. Hence, if $m$ is a positive integer, we obtain that
$$
C_n(\vert A\vert^m-\vert A^*\vert^m)C_n+(\vert A\vert^m-\vert A^*\vert^m)^*=(C_n\vert A\vert C_n)^m-\vert A^*\vert^m -(C_n\vert A^*\vert C_n)^m+\vert A\vert^m
$$
converges to zero as $n$ tends to infinity.

As $\vert A\vert^m-\vert A^*\vert^m$ is compact by Lemma \ref{cmpctdiff}, the previous lemma ensures the existence of a subsequence $\lbrace C_{m,n}\rbrace$ of $\lbrace C_n\rbrace$ so that ${P_m C_{m,n}}_{\vert\ran(P_m)}$ converges to a conjugation on $\ran(P_m)$ where $P_m$ denotes the projection onto $\overline{\ran(\vert A\vert^m-\vert A^*\vert^m)}$. Moreover, the subsequences can be chosen so that 
$$\{C_{m+1,n} : n\geq 1\}\subseteq \{C_{m,n} : n\geq 1\} \mbox{ for all } m\geq 1.$$
Applying the diagonal process argument, we obtain a subsequence $\lbrace C_{n_j}\rbrace$ such that ${P_mC_{n_j}}_{\vert\ran(P_m)}$ converges, for every $m\geq 1$, to a conjugation on $\ran(P_m)$ as $j$ tends to infinity. So if $M_0$ denotes the linear span
of all subspaces $\ran(P_m)$, $m\geq 1$, then for every $x\in M_0$, the sequence $C_{n_j}x$ converges to a vector in $M_0$. Furthermore, since $\lbrace C_{n_j}\rbrace$ is a bounded sequence, we infer that for every vector $x\in M$, the topological closure of $M_0$, $C_{n_j}x$ converges to a vector $C_Mx\in M$.

Now let us show that the conjugate-linear operator $C_M$ is a conjugation on $M$. For all $x,y\in M$, we have
$$
\langle C_Mx,C_My\rangle=\lim_{n_j}\langle C_{n_j}x,C_{n_j}y\rangle=\lim_{n_j}\langle y,x\rangle=\langle y,x\rangle,
$$
meaning that $C_M$ is isometric. Moreover, we have
\begin{eqnarray*}
\| C_M^2x -x\|=\| C_MC_Mx-x\| &=&\lim_{n_j}\|C_{n_j}C_Mx-x\|\\
&=&\lim_{n_j}\|C_{n_j}C_{n_j}C_Mx-C_{n_j}x\|\\
&=& \| C_M x-\lim_{n_j}C_{n_j}x\|=0,
\end{eqnarray*}
and so $C_M$ is a conjugation on $M$.

Note that $M=\ess(A)$. Then, for every $x\in M$, we have
\begin{eqnarray*}
\parallel\vert A\vert_{\vert M} C_Mx - C_M\vert A^*\vert_{\vert M}x\parallel &=& \parallel\vert A\vert C_Mx - C_M\vert A^*\vert x\parallel \\
&=&\lim_{n_j}\parallel \vert A\vert C_{n_j}x- C_{n_j}\vert A^*\vert x\parallel\\
&=& \lim_{n_j}\parallel C_{n_j}\vert A\vert C_{n_j}x- \vert A^*\vert x\parallel=0,
\end{eqnarray*}
whence $C_M\vert A\vert_{\vert M} C_Mx=\vert A^*\vert_{\vert M}$. Thus, $A\in\gsh$ by Theorem  \ref{ess}.
\end{proof}

The following lemma is a special case of Theorem \ref{main.t} for which the  normal operator is invertible.

\begin{lemma}
Let $N\in\lh$ be an invertible normal operator, and let $x$ and $y$ be vectors in $H$ such that
$$
\langle NE_{\vert N\vert}(\Delta)x,x\rangle=\langle NE_{\vert N\vert}(\Delta)y,y\rangle\quad\mbox{and}\quad\langle E_{\vert N\vert}(\Delta)x,x\rangle=\langle E_{\vert N\vert}(\Delta)y,y\rangle
$$
for every Borel subset $\Delta\subset\R^+$. Then
$$
N+\lambda y\otimes x \in\gsh\quad\mbox{for every $\lambda\in\C.$}
$$
\end{lemma}

\begin{proof}
Fix $\lambda\in\C$ and $\varepsilon>0$. As
$$
\| x\|^2=\langle E_{\vert N\vert}(\R^+)x,x\rangle=\langle E_{\vert N\vert}(\R^+)y,y\rangle=\| y\|^2,
$$
there is no loss of generality in assuming that $x$ and $y$ are unit vectors. Indeed, if $x=y=0$ then $N+\lambda y\otimes x=N\in\gsh$, otherwise we replace $\lambda$ by $\lambda\| x\|^{-2}$ and $x$ and $y$ by $\| x\|^{-1}x$ and $\| x\|^{-1}y$, respectively.

Using the spectral theorem, we obtain the existence of a set of pairwise disjoint intervals $\Delta_i\subset\R^+$ and positive numbers $\alpha_i \in \sigma(\vert N\vert)$ such that
$$
\parallel \vert N_i\vert-\alpha_i I \parallel \leq \varepsilon
$$
where $N_i=N_{\vert \ran(E_{\vert N\vert}(\Delta_i))}$, $1 \leq i \leq n$, and
$$
I=E_{\vert N\vert}(\Delta_1)+E_{\vert N\vert}(\Delta_2)+...+E_{\vert N\vert}(\Delta_n).
$$
Furthermore, since
$$
\vert \sqrt{x}-\sqrt{y}\vert\leq \sqrt{\| N^{-1}\|^{-1}}\vert x-y\vert\quad\mbox{for all $x,y\geq \min\sigma(\vert N\vert)$},
$$
we get by the functional calculus that
\begin{equation}\label{sqrtinneq}
\|\sqrt{\vert N_i\vert} - \sqrt{\alpha_i}I \| \leq \sqrt{\| N^{-1}\|^{-1}} \parallel\vert N_i\vert - \alpha_iI \parallel \leq \varepsilon\sqrt{\| N^{-1}\|^{-1}}
\end{equation}
for $1\leq i\leq n$. By the polar decomposition, $N_i=U_i\vert N_i\vert$ where $U_i$ is a unitary operator acting on $\ran(E_{\vert N\vert}(\Delta_i))$ and commuting with $\vert N_i\vert$.  

Let $i$ be an integer such that $1\leq i\leq n$, and write $x_i=E_{\vert N\vert}(\Delta_i)x$ and $y_i=E_{\vert N\vert}(\Delta_i)y$. Let us show that there exists a conjugation $C_i$ on $\ran(E_{\vert N\vert}(\Delta_i))$ such that $C_i \sqrt{\vert N_i\vert} x_i=\sqrt{\vert N_i\vert} y_i$ and $C_i U_i\sqrt{\vert N_i\vert} x_i= U_{i}^*\sqrt{\vert N_i\vert}y_i$. By Lemma \ref{conj.angl}, it suffices to establish the following equalities
$$\langle U_i\sqrt{\vert N_i\vert} x_i,\sqrt{\vert N_i\vert}x_i\rangle=\langle U_i\sqrt{\vert N_i\vert} y_i,\sqrt{\vert N_i\vert}y_i\rangle\quad \mbox{and}\quad 
\| \sqrt{\vert N_i\vert} x_i\|=\| \sqrt{\vert N_i\vert} y_i\|.$$
We have 
\begin{eqnarray*}
\langle U_i\sqrt{\vert N_i\vert} x_i,\sqrt{\vert N_i\vert}x_i\rangle &=&\langle \sqrt{\vert N_i\vert}U_i\sqrt{\vert N_i\vert} x_i,x_i\rangle 
=  \langle U_i\sqrt{\vert N_i\vert}\sqrt{\vert N_i\vert} x_i,x_i\rangle\\
&=& \langle U_i \vert N_i\vert x_i,x_i\rangle
= \langle  N_i x_i,x_i\rangle
= \langle  N E_{\vert N\vert}(\Delta_i)x,x\rangle,
\end{eqnarray*}
and similarly, $\langle U_i\sqrt{\vert N_i\vert} y_i,\sqrt{\vert N_i\vert}y_i\rangle=\langle  N E_{\vert N\vert}(\Delta_i)y,y\rangle$. On the other hand, since $\langle NE_{\vert N\vert}(\Delta_i)x,x\rangle=\langle NE_{\vert N\vert}(\Delta_i)y,y\rangle$, we have also
$$
\langle U_i\sqrt{\vert N_i\vert} x_i,\sqrt{\vert N_i\vert}x_i\rangle=\langle U_i\sqrt{\vert N_i\vert} y_i,\sqrt{\vert N_i\vert}y_i\rangle.
$$
To show the second equality, consider disjoint intervals $\Omega_j$ of $\R^+$, and let $\beta_j$ be positive numbers. Note that if we denote by
$
\Omega_j^2:=\lbrace t^2 : t\in\Omega_j\rbrace,
$
then it is elementary to see that $E_{A^2}(\Omega_j^2) = E_A(\Omega_j)$ for every positive operator $A$. Therefore,
\begin{eqnarray*}
\|\sum_{j} \beta_j E_{\sqrt{\vert N_i\vert}}(\Omega_j)x_i\|^2 &=&\sum_{j} \beta_j^2 \| E_{\sqrt{\vert N_i\vert}}(\Omega_j)x_i\|^2
=\sum_{j} \beta_j^2 \| E_{ \vert N_i\vert}(\Omega_j^2)x_i\|^2\\
&=&\sum_{j} \beta_j^2\langle E_{ \vert N_i\vert}(\Delta_i\cap\Omega_j^2)x_i,x_i\rangle
= \sum_{j} \beta_j^2\langle E_{ \vert N\vert}(\Delta_i\cap\Omega_j^2)x,x\rangle\\
&=& \sum_{j} \beta_j^2\langle E_{ \vert N\vert}(\Delta_i\cap\Omega_j^2)y,y\rangle
= \|\sum_{j} \beta_j E_{\sqrt{\vert N_i\vert}}(\Omega_j)y_i\|^2.
\end{eqnarray*}
Since $\sqrt{\vert N_i\vert}$ is limit of operators of the form $\sum \beta_j E_{\sqrt{\vert N_i\vert}}(\Omega_j)$, we obtain that $\| \sqrt{\vert N_i\vert} x_i\|=\| \sqrt{\vert N_i\vert} y_i\|$, as desired.

Put
$$
M=\begin{pmatrix}
\alpha_1 U_1 & & &\\
& \alpha_2 U_2 & &\\
& & \ddots & \\
& & & \alpha_n U_n 
\end{pmatrix}
\begin{array}{l}
\ran(E_{\vert N\vert}(\Delta_1))\\
\ran(E_{\vert N\vert}(\Delta_2))\\
\vdots\\
\ran(E_{\vert N\vert}(\Delta_n))
\end{array},
$$
$$
a=\oplus_{i=1}^n \alpha_i^{-1/2}\sqrt{\vert N_i\vert} x_i,\quad b=\oplus_{i=1}^n \alpha_i^{-1/2}\sqrt{\vert N_i\vert} y_i\quad\mbox{and}\quad C=\oplus_{i=1}^n C_i.
$$
We have $\|a\|=\| b\|$, and one can easily check that $C$ is a conjugation on $H$ that satisfies 
$$
CMM^*C=M^*M,\quad Cb=a\quad \mbox{and}\quad CM a=M^*b.
$$ 
Hence, it follows that
\begin{eqnarray*}
C(M+\lambda b\otimes a)(M+\lambda b\otimes a)^*C &=& C(M+\lambda b\otimes a)(M^*+\overline{\lambda} a\otimes b)C\\
&=& C( MM^* + \overline{\lambda}Ma\otimes b + \lambda b\otimes Ma
\\
& & +\vert\lambda\vert^2\| a\|^2 b\otimes b)C\\
&=&  CMM^*C + \lambda C Ma\otimes Cb + \overline{\lambda} Cb\otimes CMa\\
& & +  \vert\lambda\vert^2\| b\|^2 Cb\otimes Cb\\
&=& M^*M + \lambda M^*b\otimes a +  \overline{\lambda} a\otimes M^*b\\
& & + \vert\lambda\vert^2\| b\|^2 a\otimes a\\
&=& (M+\lambda b\otimes a)^*(M+\lambda b\otimes a),
\end{eqnarray*}
meaning that the operator $M+\lambda b\otimes a$ is $C$-normal. Moreover, we have
\begin{eqnarray*}
\| a-x\|=\parallel (\oplus_{i=1}^n \alpha_i^{-1/2}\sqrt{\vert N_i\vert}) x - x \parallel &\leq & \| (\oplus_{i=1}^n \alpha_i^{-1/2}\sqrt{\vert N_i\vert}) - I \|\\
&\leq & \max_{1\leq i\leq n}\| \alpha_i^{-1/2}\sqrt{\vert N_i\vert} - I \|\\
&\leq & \max_{1\leq i\leq n}\alpha_i^{-1/2}\| \sqrt{\vert N_i\vert} - \sqrt{\alpha_i}I \|,
\end{eqnarray*}
and so we obtain, by (\ref{sqrtinneq}), that
$$
\| a-x\|\leq \varepsilon\sqrt{\| N^{-1}\|^{-1}}\max_{1\leq i\leq n}\sqrt{\alpha_i^{-1}}\leq \varepsilon\sqrt{\parallel N^{-1}\parallel}\sqrt{\| N^{-1}\|^{-1}}=\varepsilon.
$$
Similarly, we get that $\| b-y\|\leq\varepsilon$. Thus
\begin{eqnarray*}
\| N+\lambda y\otimes x - (M + \lambda b\otimes a)\| &\leq & \| N-M\| +\vert\lambda\vert.\| y\otimes x-b\otimes a\|\\
&\leq & \varepsilon + \vert\lambda\vert\left(\| y\otimes x-y\otimes a\|+\| y\otimes a-b\otimes a\|\right)\\
&\leq & \varepsilon + \vert\lambda\vert \left(\| x-a\|+\| y-b\|.\| a\| \right)\\
&\leq & \varepsilon + \vert\lambda\vert \left(\| x-a\|+\| y-b\|(\| x-a\|+1) \right)\\
&\leq & \varepsilon +\varepsilon\vert\lambda\vert+\varepsilon^2\vert\lambda\vert+\varepsilon\vert\lambda\vert.
\end{eqnarray*}
As $\varepsilon$ is arbitrary, we obtain that $N+\lambda y\otimes x\in\overline{\gsh}$. Hence, since the operator $N+\lambda y\otimes x$ is essentially normal, Lemma \ref{essnor} implies that $N+\lambda y\otimes x\in\gsh$.
\end{proof}

We are now in a position to prove Theorem \ref{main.t}.

\begin{proof}[Proof of Theorem \ref{main.t}]
In view of the previous lemma, we need only to consider the case that $N$ is not invertible. Let $n$ be a positive integer, and let $R_n$ be the set given by
$$
R_n=\lbrace z\in\C : \vert z\vert > n^{-1}\rbrace.
$$
 Consider the operator $N_n$ given by
$$
N_n=\begin{pmatrix}
N_{\vert \ran(E_N(R_n))} & 0\\
0 & (2n)^{-1}
\end{pmatrix}
\begin{array}{l}
\ran(E_N(R_n))\\
\ran(E_N(\C\setminus R_n))
\end{array}.
$$
Then $N_n$ is an invertible normal operator. Let $\Delta\subset \R^+$ be a Borel subset. We have
\begin{eqnarray*}
\langle N_n E_{\vert N_n\vert}(\Delta) x,x\rangle &=&\left\langle N_n \left(E_{\vert N_n\vert}(\Delta\cap R_n)+E_{\vert N_n\vert}(\Delta\cap(\C\setminus R_n)\right) x,x\right\rangle\\
&=& \left\langle N_n E_{\vert N_n\vert}(\Delta\cap R_n) x,x\rangle+\langle N_n E_{\vert N_n\vert}(\Delta\cap(\C\setminus R_n)) x,x\right\rangle\\
&=&\left\langle N E_{\vert N\vert}(\Delta\cap R_n) x,x\right\rangle+(2n)^{-1} \left\langle E_{\vert N_n\vert}(\Delta\cap\lbrace (2n)^{-1}\rbrace) x,x\right\rangle.
\end{eqnarray*}
But, $E_{\vert N_n\vert}(\Delta\cap\lbrace (2n)^{-1}\rbrace)=E_{\vert N\vert}(\C\setminus R_n)$ if $(2n)^{-1}\in\Delta$, and $E_{\vert N_n\vert}(\Delta\cap\lbrace (2n)^{-1}\rbrace)=0$ otherwise. Hence
\begin{eqnarray*}
\langle N_n E_{\vert N_n\vert}(\Delta) x,x\rangle &=&\langle N E_{\vert N\vert}(\Delta\cap R_n) y,y\rangle+(2n)^{-1}\langle E_{\vert N_n\vert}(\Delta\cap\lbrace (2n)^{-1}\rbrace) y,y\rangle\\
&=& \langle N_n E_{\vert N_n\vert}(\Delta) y,y\rangle.
\end{eqnarray*}
Similarly, we obtain that
$
\langle E_{\vert N_n\vert}(\Delta) x,x\rangle=\langle E_{\vert N_n\vert}(\Delta) y,y\rangle.
$
Hence, by the previous lemma, $N_n+\lambda y\otimes x\in\gsh$ for every $\lambda\in\C$.

Finally, since
$$
N-N_n=\begin{pmatrix}
0 & 0\\
0 & N_{\vert \ran(E_N(\C\setminus R_n))} - (2n)^{-1}
\end{pmatrix}
$$
and 
$$
\| N-N_n\|=\|  N_{\vert \ran(E_N(\C\setminus R_n))} - (2n)^{-1}\|\leq \|  N_{\vert \ran(E_N(\C\setminus R_n))}\|+(2n)^{-1}\leq n^{-1}+(2n)^{-1} =3(2n)^{-1},
$$
Lemma \ref{essnor} implies that $N+\lambda y\otimes x\in\gsh$, which completes the proof.
\end{proof}

\end{document}